\documentclass[12pt,a4paper]{article}

\usepackage[T1]{fontenc}
\usepackage[english]{babel}
\usepackage[vlined,linesnumbered]{algorithm2e}
\usepackage{amsmath}
\usepackage{amssymb}
\usepackage{amsthm}
\usepackage{booktabs}
\usepackage{graphicx}
\usepackage{hyperref}
\usepackage{microtype}
\usepackage{subcaption}

\newtheorem{theorem}{Theorem}

\DeclareMathOperator{\sign}{sgn}

\makeatletter
\newcommand{\STS@text}[1]{{\rm STS}$(#1)$}
\newcommand{\STS@unstarred}[1]{\ifmmode{\text{\mbox{\STS@text{#1}}}}\else{\mbox{\STS@text{#1}}}\fi}
\newcommand{\STS@starred}[1]{\ifmmode{\text{\mbox{sub-\STS@text{#1}}}}\else{\mbox{sub-\STS@text{#1}}}\fi}
\newcommand{\STSS@unstarred}[1]{\ifmmode{\text{\mbox{\STS@text{#1}s}}}\else{\mbox{\STS@text{#1}s}}\fi}
\newcommand{\STSS@starred}[1]{\ifmmode{\text{\mbox{sub-\STS@text{#1}s}}}\else{\mbox{sub-\STS@text{#1}s}}\fi}
\newcommand{\STS}{\@ifstar{\STS@starred}{\STS@unstarred}}
\newcommand{\STSS}{\@ifstar{\STSS@starred}{\STSS@unstarred}}
\makeatother

\newcommand{\ConfPasch}{\includegraphics[page=1]{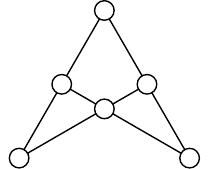}}
\newcommand{\ConfMitre}{\includegraphics[page=2]{images.pdf}}
\newcommand{\ConfFanoLine}{\includegraphics[page=3]{images.pdf}}
\newcommand{\ConfCrown}{\includegraphics[page=4]{images.pdf}}
\newcommand{\ConfHexagon}{\includegraphics[page=5]{images.pdf}}
\newcommand{\ConfPrism}{\includegraphics[page=6]{images.pdf}}
\newcommand{\ConfGrid}{\includegraphics[page=7]{images.pdf}}
\newcommand{\ConfFano}{\includegraphics[page=8]{images.pdf}}
\newcommand{\ConfMobiusKantor}{\includegraphics[page=9]{images.pdf}}

\newcommand{\ImgPasch}{\includegraphics[page=10]{images.pdf}}
\newcommand{\ImgMitre}{\includegraphics[page=11]{images.pdf}}
\newcommand{\ImgFanoLine}{\includegraphics[page=12]{images.pdf}}
\newcommand{\ImgCrown}{\includegraphics[page=13]{images.pdf}}
\newcommand{\ImgHexagon}{\includegraphics[page=14]{images.pdf}}
\newcommand{\ImgPrism}{\includegraphics[page=15]{images.pdf}}
\newcommand{\ImgGrid}{\includegraphics[page=16]{images.pdf}}
\newcommand{\ImgFano}{\includegraphics[page=17]{images.pdf}}
\newcommand{\ImgMobiusKantor}{\includegraphics[page=18]{images.pdf}}

\begin{document}

\title{Constructing Random Steiner Triple Systems: An Experimental Study}

\author{
Daniel Heinlein\thanks{Supported by the Academy of Finland, Grant 331044.}\phantom{ }
and Patric R. J. \"Osterg\aa rd\\
Department of Information and Communications Engineering\\
Aalto University School of Electrical Engineering\\
P.O.\ Box 15400, 00076 Aalto, Finland\\
\tt \{daniel.heinlein,patric.ostergard\}@aalto.fi
}

\date{}

\maketitle

\begin{center}{\sl Dedicated to Doug Stinson on the occasion of his 66th birthday.}\end{center}

\abstract{Several methods for generating random Steiner triple systems (STSs) have been proposed
in the literature, such as Stinson's hill-climbing algorithm and Cameron's algorithm, but
these are not yet completely understood.
Those algorithms,
as well as some variants, are here assessed for STSs of both small and large orders.
For large orders, the number of occurrences of certain
configurations in the constructed STSs are compared
with the corresponding expected values of random hypergraphs.
Modifications of the algorithms are proposed.}

\section{Introduction}\label{sec:introduction}

Random generation of combinatorial structures is a problem that has been extensively
studied both from a theoretical~\cite{J,JVV} and a practical~\cite{C0,HDM,S0} point of view.
When developing practical algorithms for random generation of structures with tight constraints,
it can be far from clear how amenable central measures such as computation
time and uniformity are to formal study. For the smallest parameters, structures are completely
understood, but both the algorithms and the distribution of structures
can behave differently for small parameters than for large ones.
On the other hand, for large parameters where the algorithms are truly needed,
it is challenging to evaluate the performance.

Steiner triple systems (STSs) form one class of structures for which practical random
generation algorithms have been developed but where an evaluation for large
parameters is lacking, primarily due to the difficulty of judging performance.
In this experimental study, we start by considering the case of Steiner triple
systems with small orders and carry out evaluations of known algorithms,
including some variants. For larger orders, we bring together random generation
algorithms and a conjecture on substructures of Steiner triple systems obtained
using a random model.
In this paper, the desired distribution for the random generation of STSs is uniform among all
\emph{labeled STSs} of a fixed order.

In 1985, Stinson~\cite{S0} published his seminal \emph{hill-climbing} approach to generate random STSs.
By exploring generated STSs with order 15, Stinson observed that STSs with large automorphism groups are underrepresented.
As noted in~\cite{S0}, modifications of the algorithm are able to create Latin squares and strong starters and to complete partial STSs to STSs.

Heap, Danziger, and Mendelsohn~\cite{HDM} studied the number of occurrences of substructures in STSs generated with Stinson's algorithm and reported that STSs with many Pasch configurations
are underrepresented for orders up to 19
and STSs with many mitre configurations
are overrepresented for order 15.
Two variations of Stinson's algorithm based on performing Pasch trades are
also proposed in~\cite{HDM}.

Random generation of Latin squares based on \emph{Markov chains} by Jacobson and
Matthews~\cite{JM} inspired Cameron~\cite{C0} to design a related algorithm for STSs.
We here call the algorithm for STSs \emph{Cameron's algorithm}, but recognize that it could
also be called the \emph{Jacobson--Matthews--Cameron algorithm}.

The paper is organized as follows. In Section~\ref{sec:preliminaries},
Steiner triple systems and related structures and concepts are defined,
and the topic of generating random Steiner triple systems is introduced.
In Section~\ref{sec:algo}, Stinson's algorithm and Cameron's algorithm,
as well as some variants, are described. Finally,
computational experiments are carried out and analyzed in Section~\ref{sec:results}
and the paper is concluded in Section~\ref{sec:conclusions}.

\section{Preliminaries}\label{sec:preliminaries}

\subsection{Steiner Triple Systems and Configurations}\label{sect:switch}

A \emph{Steiner triple system} (STS) is a pair $(V,\mathcal{B})$,
where $V$ is a set of \emph{points} and $\mathcal{B}$ is a
set of \mbox{3-subsets} of points, called \emph{blocks}, such that
every \mbox{2-subset} of points occurs in exactly one block. The size of
the point set is the \emph{order} of the Steiner triple system, and a
Steiner triple system of order $v$ is denoted by \STS{v}.
An \STS{v} exists iff
\begin{align*}
v \equiv 1\text{ or }3\!\!\!\pmod{6}.
\end{align*}
Direct calculation gives that an
\STS{v} has $v(v-1)/6$ blocks and each point is in $(v-1)/2$ blocks.
More information about Steiner triple systems can be found in~\cite{C,CR}.

A \emph{transversal design}
$\operatorname{TD}(k,n)$
is a triple
$(V,\mathcal{G},\mathcal{B})$, where $V$ is a set of $kn$ elements;
$\mathcal{G}$ is a partition of $V$ into $k$ \mbox{$n$-subsets} (called \emph{groups});
$\mathcal{B}$ is a collection of \mbox{$k$-subsets} of $V$ (called \emph{blocks});
and every \mbox{2-subset} of $V$ is contained either in exactly one group
or in exactly one block, but not both.

A \emph{configuration} is a pair $(V',\mathcal{B'})$, where
$V'$ is a set of points and $\mathcal{B'}$ is a collection of
subsets of points---typically called \emph{lines} but here called
\emph{blocks} to comply with other definitions---such that
every 2-subset of points occurs in at most one block.

A configuration with the property that each point in
$V'$ occurs in at least two blocks is said to be \emph{full}.
A configuration with $w$ points, $w$ blocks, each block
containing $k$ points, and each point occurring in
$k$ blocks is called a $w_k$ configuration.  Throughout the paper,
we assume block size 3.
In the context of building up an STS block by block, the configuration
consisting of the blocks that have been obtained so far is called
a \emph{partial STS}. It is possible that a partial STS cannot
be completed to an STS.

The configurations considered in this paper are precisely the
full $n$-block configurations with $n \le 6$ and the $w_3$ configurations with
$w \le 8$, all of which are depicted in Figure~\ref{fig:all_configurations}.

\begin{figure}[htbp]
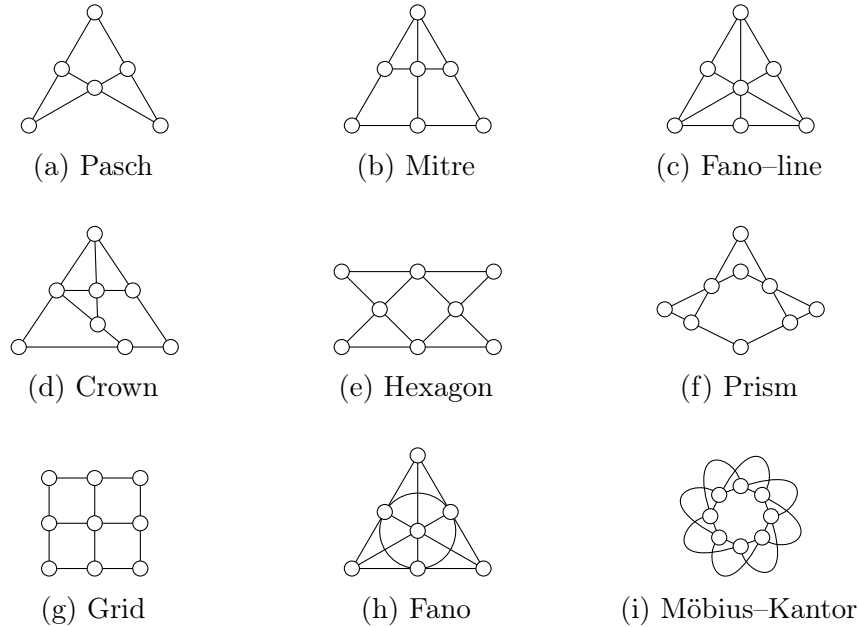

\vspace{5mm}
\centering
\begin{subfigure}[t]{0.3\textwidth}
\centering
\ConfPasch
\caption{Pasch}\label{sf:pasch}
\end{subfigure}
\begin{subfigure}[t]{0.3\textwidth}
\centering
\ConfMitre
\caption{Mitre}\label{sf:mitre}
\end{subfigure}
\begin{subfigure}[t]{0.3\textwidth}
\centering
\ConfFanoLine
\caption{Fano--line}\label{sf:fanoline}
\end{subfigure}
\vskip 0.5cm
\begin{subfigure}[t]{0.3\textwidth}
\centering
\ConfCrown
\caption{Crown}\label{sf:crown}
\end{subfigure}
\begin{subfigure}[t]{0.3\textwidth}
\centering
\ConfHexagon
\caption{Hexagon}\label{sf:hexagon}
\end{subfigure}
\begin{subfigure}[t]{0.3\textwidth}
\centering
\ConfPrism
\caption{Prism}\label{sf:prism}
\end{subfigure}
\vskip 0.5cm
\begin{subfigure}[t]{0.3\textwidth}
\centering
\ConfGrid
\caption{Grid}\label{sf:grid}
\end{subfigure}
\begin{subfigure}[t]{0.3\textwidth}
\centering
\ConfFano
\caption{Fano}\label{sf:fano}
\end{subfigure}
\begin{subfigure}[t]{0.3\textwidth}
\centering
\ConfMobiusKantor
\caption{M\"{o}bius--Kantor}\label{sf:mobiuskantor}
\end{subfigure}
\caption{Configurations}\label{fig:all_configurations}
\end{figure}

An \emph{isomorphism} between two STSs is a bijection between the point sets that preserves incidences.
An isomorphism of an STS onto itself is an \emph{automorphism}, and the set of all automorphisms of an STS forms a group under composition.
The concepts of isomorphism, automorphism, and automorphism group are similarly
defined for configurations.

A switch or a trade can be used to modify a Steiner triple system into
another Steiner triple system of the same order~\cite{O}. A \emph{cycle switch} of a Steiner
triple system is defined as follows. Suppose that an STS contains the blocks
\begin{equation}\label{eq:switch}
T_1 = \{a x_1 x_2,b x_2 x_3,a x_3 x_4,b x_4 x_5,\ldots,b x_n x_1\}.
\end{equation}
Then these blocks can be replaced by the blocks
\[
T_2 = \{b x_1 x_2,a x_2 x_3,b x_3 x_4,a x_4 x_5,\ldots,a x_n x_1\}.
\]
to get another STS.

For fixed and distinct points $a$ and $b$, and $c$ chosen so
that $\{a,b,c\} \in \mathcal{B}$, the points $V \setminus \{a,b,c\}$
can in a unique way be partitioned into sets of points that occur together with
$a$ and $b$ in sets of type $T_1$. The name
cycle switch comes from the fact that
removal of the points $a$ and $b$ from the blocks
in $T_1$ gives edges that form a cycle in a graph with vertex set
$\{x_1,x_2,\ldots,x_n\}$.

The smallest possible size of $T_1$ and $T_2$ is 4, and such a switch
is called a \emph{Pasch switch} as $T_1$ and $T_2$ are then Pasch
configurations.

A \emph{perfect} \STS{v} is a Steiner triple system, in which each possible cycle
switch involves $|T_1|=|T_2|=v-3$ blocks. Perfect Steiner triple systems indeed
exist; see, for example,~\cite{GGM}.

When carrying out cycle switching for partial STSs, one also encounters
situations where the two points $a$ and $b$ do not induce a 2-regular graph
with vertex set $V \setminus \{a,b,c\}$, but induces a graph with vertex
degrees at most 2. Such a graph consists of cycles and paths, and switches
for paths are analogous to those for cycles. For example, the switch for the
shortest possible path (of length 1) is simply
\begin{equation}\label{eq:path}
\{a x_1 x_2\} \rightarrow \{b x_1 x_2\}.
\end{equation}

\subsection{Hypergraph Models}

Random \mbox{3-uniform} hypergraphs on $v$ vertices play a central role in our
experimental study of algorithms for constructing random Steiner triple systems of order $v$.
Specifically, we will focus on numerical values related to properties of the STSs and
compare them with the corresponding values determined for the hypergraph models.
Lacking theoretical results, computational results indicating that such values
have the same asymptotic behavior provide some evidence for quality of
the algorithms as well as the model with respect to the studied properties.

We denote the number of blocks in an \STS{v} by $b=v(v-1)/6$ and
the number of 3-subsets of the set of $v$ vertices by $w=\binom{v}{3}$,
The properties of STSs to be considered are specifically the numbers of
occurrences of configurations. For a given configuration with $v'$ points,
$b'$ blocks, and automorphism group order $|\Gamma|$, there are $v'! / |\Gamma|$
labeled specimens on $v'$ points and consequently
\begin{equation}
\label{eq:zero}
\frac{v'!\binom{v}{v'}}{|\Gamma|} = \prod_{i=0}^{v'-1} \frac{v-i}{|\Gamma|} \sim \frac{v^{v'}}{|\Gamma|}
\end{equation}
labeled specimens on $v$ points.

We consider generalizations of the two
most common models for random graphs. In the first model,
each hyperedge appears with equal probability $p$. With $p =  b/w=1/(v-2)$, the expected
number of hyperedges is $b$, and this model has been used for STSs in~\cite{K0,HO}.
In this model, the probability for a specific specimen of a configuration to occur is
\begin{equation}
\label{eq:one}
p^{b'} = (v-2)^{-b'} \sim v^{-b'}
.
\end{equation}

In the second model, a random hypergraph is chosen uniformly from the set of
all hypergraphs with $b$ hyperedges. Now the probability for a specific specimen
of a configuration to occur is
\begin{equation}
\label{eq:two}
\frac{\binom{w-b'}{b-b'}}{\binom{w}{b}}
=
\prod_{i=0}^{b'-1}\frac{b-i}{w-i}
=
\prod_{i=0}^{b'-1}\frac{v^2-v-6i}{v^3-3v^2+2v-6i}
\sim
v^{-b'}.
\end{equation}

By~\eqref{eq:zero},~\eqref{eq:one},~\eqref{eq:two}, and
linearity of expectation, the asymptotic expected number
of occurrences of the configuration is for both models
\begin{equation}
\label{eq:asymptotic}
\frac{v'! \binom{v}{v'}v^{-b'}}{|\Gamma|}
\sim
\frac{v^{v'-b'}}{|\Gamma|}.
\end{equation}

For the configurations in Figure~\ref{fig:all_configurations}, we list
in Table~\ref{tab:configuration_details} the number of blocks $b'$,
the number of points $v'$,
the order of the automorphism group $|\Gamma|$, and the right-hand side
of~\eqref{eq:asymptotic}.

\begin{table}[htbp]
\centering
\caption{Configuration data}\label{tab:configuration_details}
\begin{tabular}{lrrrr}\toprule
Name & $b'$ & $w'$ & $|\Gamma|$ & asymptotic behavior \\\midrule
Pasch & 4 & 6 & 24 & $v^2/24$ \\
Mitre & 5 & 7 & 12 & $v^2/12$ \\
Fano--line & 6 & 7 & 24 & $v/24$ \\
Crown & 6 & 8 & 2 & $v^2/2$ \\
Hexagon & 6 & 8 & 12 & $v^2/12$ \\
Prism & 6 & 9 & 12 & $v^3/12$ \\
Grid & 6 & 9 & 72 & $v^3/72$ \\\midrule
Fano & 7 & 7 & 168 & $1/168$ \\
M\"{o}bius--Kantor & 8 & 8 & 48 & $1/48$ \\\bottomrule
\end{tabular}
\end{table}

A hypergraph in which every pair of vertices is contained in at most one hyperedge
is called linear. The probability that a random linear uniform
hypergraph with a given number of vertices and hyperedges
contains a given hypergraph as a subhypergraph is determined in~\cite[Theorem~1.4]{MT},
but for random 3-uniform hypergraphs the result only applies to the case of $o(v^{3/2})$ hyperedges.

\section{Algorithms for Generating Random STSs}\label{sec:algo}

Now we shall have a look at Stinson's algorithm and Cameron's algorithm. A
key difference between these is that an STS is constructed from scratch
in Stinson's algorithm, whereas Cameron's algorithm repeatedly perturbs a
structure that is an STS or nearly an STS (in a way that will become clear
later). Only for Stinson's algorithm do
we know that every STS can be reached with a non-zero probability.

\subsection{Random Steiner Triple Systems}\label{sec:random}

We would like to develop an algorithm that constructs random Steiner triple
system of a given order, with uniform distribution over all labeled systems
of that order. If the Steiner triple systems of the given order have been
classified, then there is an obvious solution. Pick a random isomorphism
class, where the probability of choosing a class is given by the proportion of
STSs in that class, and then apply a random permutation to the point set of
the chosen isomorphism class representative. Since Steiner triple systems
have been classified for all orders up to 19, this approach does not work for
orders greater than or equal to 21. Moreover, a rather huge data structure is needed
for the more than 11 billion isomorphism classes of \STSS{19}~\cite{KO1}.
(The problem of picking an isomorphism class uniformly
at random is obviously straightforward for these parameters.)

What about orders greater than 19? In theory, the problem can be solved
for arbitrary orders: the classification problem can be solved in finite
time, whereafter the above mentioned approach can be used. Another solution,
with finite expected time but infinite worst-case time, is as follows.
Take a random $(v(v-1)/6)$-subset of blocks and check whether these form an STS. If not,
repeat the procedure.
(For the problem of picking isomorphism classes uniformly
at random, accept an STS that has been found with the probability given by
the size of the automorphism group of the STS divided by a fixed constant that is
an upper bound on the size of the automorphism groups of the STSs of the
same order). Because of implementation and/or time issues,
these methods have no practical importance.

We will next consider some methods that have been proposed for the
construction of random labeled Steiner triple systems. How to evaluate them?
Of course, one could evaluate them for small cases---up to order 19---which
are completely understood. Such work is done in~\cite{HDM}. But what if the
behavior of an algorithm changes when going from small to large orders?
One challenge when evaluating algorithms for large orders is that such
an evaluation is necessarily based on properties of the STSs and the
behavior of a property when going from small to large orders is not
necessarily well understood.

\subsection{Stinson's Algorithm}

Stinson~\cite{S0} developed a celebrated hill-climbing algorithm for
constructing \STSS{v}. This algorithm is presented as Algorithm~\ref{algo:stinson}.

\vspace{7mm}
\begin{algorithm}[H]
\DontPrintSemicolon{}
$\mathcal{B} \leftarrow \emptyset$\;
\While{$|\mathcal{B}| < v(v-1)/6$}{
choose a point $x$ that appears in fewer than $(v-1)/2$ blocks of $\mathcal{B}$\;
choose a point $y$ such that $\{x,y\}$ is not a subset of any block of $\mathcal{B}$\;
choose a point $z\ne y$ such that $\{x,z\}$ is not a subset of any block of $\mathcal{B}$\;
\If{there is a block $B \in \mathcal{B}$ such that $\{y,z\} \subseteq B$}{
$\mathcal{B} \leftarrow \mathcal{B} \setminus \{B\}$\;
}
$\mathcal{B} \leftarrow \mathcal{B} \cup \{\{x,y,z\}\}$\;
}
\vspace*{5mm}
\caption{Stinson's hill-climbing algorithm}\label{algo:stinson}
\end{algorithm}

\vspace{7mm}

The distribution of STSs generated by Stinson's algorithm is not known. Experiments
showing that the distribution is not uniform for small orders~\cite{S0,HDM} have perhaps
discouraged people from further investigation. However, as argued in Section~\ref{sec:random},
the situation may change with growing order, and we shall later
elaborate on this.

There exist situations where Stinson's algorithm does not terminate, as shown by
the example in~\cite[p.~37]{CR} for \STSS{15}. That example can easily be generalized to
an infinite family.

\begin{theorem}\label{thm:stuck}
Let $v=6n+3$ with $n \equiv 2 \pmod{3}$. Then there is a non-zero probability
that an execution of Stinson's algorithm in the search for an \STS{v} will
not terminate.
\end{theorem}

\begin{proof}
For any partial STS, there is a non-zero probability of encountering it
during the execution of the algorithm. Consider the situation where
the partial STS is a $\operatorname{TD}(3,2n+1)$ transversal design.
After choosing $x$, the points
$y$ and $z$ necessarily come from the same group as $x$ and hence no blocks of the transversal
design are ever removed. As the number of 2-subsets of points in a group is
$2n^2+n$, which is not divisible by 3 when $n \equiv 2 \pmod{3}$, the partial STS cannot
be completed.
\end{proof}

Although rare, situations of this kind are
encountered in practice.
An obvious way of dealing with a situation such as that in Theorem~\ref{thm:stuck}
is to set a limit on
the number of executions of the \textbf{while} loop~\cite{S0,CR,HDM}. Obviously, at least $b=v(v-1)/6 = \Theta(v^2)$ iterations are needed with Stinson's algorithm.
The data structure for maintaining partial STSs can
be implemented so that execution of any line in Algorithm~\ref{algo:stinson}
takes constant time, cf.~\cite{S0}.
There is empirical evidence~\cite{S0} that Stinson's algorithm with a limit on the
number of steps needs $\Theta(v^2 \log v)$ time on the average to construct STSs.

\subsection{Modifications of Stinson's Algorithm}

One weakness of Stinson's algorithm is that the distribution of
constructed STSs is not uniform for small orders. In particular, \STSS{v} with a large
number of Pasch configurations are underrepresented. Heap, Danziger, and Mendelsohn~\cite{HDM}
therefore suggest an extension of Stinson's algorithm, where Pasch switches are
occasionally carried out: If the constructed STS contains at least one Pasch configuration,
then with a fixed probability a Pasch switch is carried out on a (random) Pasch configuration
and the STS thereby obtained is also output. (A related algorithm
that produces single STSs is also presented in~\cite{HDM}, but that algorithm actually
produces STSs with a different distribution; for example, STSs without Pasch configurations
are more frequent.)

In the current work, experiments were made to try to find variants of Stinson's algorithm
that would lead to a more uniform distribution of constructed STSs. Whereas the approach in~\cite{HDM}
is to modify complete STSs, the main focus here is on modifying the search algorithm itself.

Choices are made in three places in Algorithm~\ref{algo:stinson}, in lines~3 to~5.
When implementing the algorithm, one needs to specify how these choices are made.
Such details are often omitted in studies where the algorithm has been
used. One reason behind this is that statements like ``choose'' and ``choose randomly''
can be understood as choosing uniformly among candidates.
In any case, precision is essential when analyzing and comparing
variants of the algorithm.

Possible parameters for deriving probability distributions for the choice of
the point $x$ in line~3 of Algorithm~\ref{algo:stinson} are $n_q$, defined
as the number of blocks in which the point $q \in V$ already appears, and
$m_q = (v-1)/2-n_q$. For any function $f_x(i)$ with nonnegative values,
a probability distribution can now be obtained as

\begin{equation}\label{eq:choice}
p_x(q) = \frac{f_x(m_q)}{\sum_{j\in V} f_x(m_j)},
\end{equation}

\noindent
where $f_x(m_q)$ must be positive for at least one point $q$.
For example, uniform distribution over candidate points is obtained
with the signum function, $f_x(i) = \sign i$.

Obviously, nonuniform distributions may also be used for choosing
$y$ and $z$ in Algorithm~\ref{algo:stinson}. The algorithms obtained
with various distributions for $x$, $y$, and $z$ are here called
\emph{weighted Stinson's algorithms}.

An alternative possibility of changing Stinson's algorithm is to
change lines~6 to~8 of Algorithm~\ref{algo:stinson}. In that part of
the algorithm, a block
$B' = \{x,y,z\}$ is added and a block $B = \{y,z,w\}$ is possibly removed.
The situation where $B$ is removed is precisely the switch~\eqref{eq:path},
as noticed in \mbox{\cite[p.~629]{O}}.

A switch in Stinson's algorithm---removing one block and adding
another---can be done in constant time. By changing this part of the algorithm,
one may end up trading speed for a better distribution of constructed STSs. We
call algorithms where switching is done in different ways
\emph{extended Stinson's algorithms}.

\subsection{Cameron's Algorithm}\label{sec:cameron_alg}

Jacobson and Matthews~\cite{JM} published an algorithm for random generation of Latin
squares, and Cameron~\cite{C0,C1,C2} showed that a similar approach is possible
also for Steiner triple systems.
The main difference
between the two settings is that stronger theoretical results have been proved for the
Latin square algorithm.

The idea in Cameron's algorithm is to apply perturbations repeatedly to get a sequence
of structures, some of which are STSs and
some of which are not (and are called proper and improper STSs, respectively).
These structures can be defined as pairs $(V,f)$, where $V$ is a set
of points and $f$ is a function that maps \mbox{3-subsets} of points
into the set $\{-1,0,1\}$ such that at most one 3-subset is mapped
to $-1$ and
\begin{align*}
\sum_{z \in V \setminus \{x,y\}} f(\{x,y,z\}) = 1
\end{align*}
for all \mbox{2-subsets} $\{x,y\}$. We say that we have
\begin{itemize}
\item a \emph{proper} STS if $f(B) \ne -1$ for all \mbox{3-subsets} $B$,
\item an \emph{improper} STS if there is a
unique \mbox{3-subset} $B'$ such that $f(B') = -1$.
\end{itemize}
An STS as defined in Section~\ref{sec:preliminaries} is
a proper STS, where $f$ is the characteristic function of the
set of blocks. For an improper STS---which has $f(B')=-1$---the
3-subsets $B$ for which $f(B)=1$ correspond to a triple system
where each 2-subset of points occur in exactly one block, except for
the 2-subsets of $B'$, which occur
in exactly two blocks.

The possible perturbations are described in Table~\ref{tab:cameron} for
proper and improper STSs. In each case, the candidates are given by
the values of $x$, $y$, $z$, $x'$, $y'$, and $z'$ that fulfill the
conditions in line~$f$, and the new function values are given in
line~$f'$. The total number of candidates is $v(v-1)(v-3)/6$ for proper STSs
and $8$ for improper STSs~\cite{C0}. One of these is chosen uniformly at
random in the algorithm. The value of $f'(\{x',y',z'\})$ determines whether
the new structure is a proper or improper STS. Note that these transformations
are essentially about doing Pasch switches. The random walk can be implemented
so that each step takes constant time.

\begin{table}[htbp]
\caption{Perturbations for proper and improper STSs}\label{tab:cameron}
\setlength{\tabcolsep}{1pt}
\begin{center}
\begin{tabular}{r|c|ccc|ccc|c}
\multicolumn{9}{c}{Proper}\\\toprule
& $\{x,y,z\}$ & $\{x',y,z\}$ & $\{x,y',z\}$ & $\{x,y,z'\}$ & $\{x,y',z'\}$ & $\{x',y,z'\}$ & $\{x',y',z\}$ & $\{x',y',z'\}$\\\midrule
     $f$ & $0$ & $1$ & $1$ & $1$ & $0$ & $0$ & $0$ & $a \geq 0$ \\
      $f'$ & $1$ & $0$ & $0$ & $0$ & $1$ & $1$ & $1$ & $a-1$ \\\bottomrule
\multicolumn{9}{c}{}\\
\multicolumn{9}{c}{Improper}\\\toprule
& $\{x,y,z\}$ & $\{x',y,z\}$ & $\{x,y',z\}$ & $\{x,y,z'\}$ & $\{x,y',z'\}$ & $\{x',y,z'\}$ & $\{x',y',z\}$ & $\{x',y',z'\}$\\\midrule
         $f$ & $-1$ & $1$ & $1$ & $1$ & $0$ & $0$ & $0$ & $a \geq 0$ \\
         $f'$ & $0$ & $0$ & $0$ & $0$ & $1$ & $1$ & $1$ & $a-1$ \\\bottomrule
\end{tabular}
\end{center}
\end{table}

A directed graph in which the vertices are combinatorial structures and the arcs show
possible transformations from one structure into another is called a \emph{transition
graph}. Reversible transformations can be modeled with undirected graphs.
Here we may consider an undirected transition graph $G$ whose vertices are the proper
and improper STSs. Repeated perturbations that are carried out uniformly at random correspond to a
random walk (Markov chain) in $G$. It can be shown~\cite{C0}
that the unique limiting distribution of this Markov chain has the property that
all proper
Steiner triple systems in the connected component where the walk starts
have equal probability in the stationary distribution.
The central, still open, question is whether $G$ is connected; this
is the case for small orders of Steiner triple systems~\cite{C0,DGG}.

As we are only interested in proper STSs, it is useful to notice that
the subsequence of proper STSs encountered in the walk is also a Markov chain with a unique stationary
distribution that is uniform over the set of proper STSs in the same connected component; this follows
from the proof of~\cite[Theorem~4]{JM}.

Cameron's algorithm has two central parameters. First, one needs to choose the starting
point of the random walk. One can then use an STS coming from some construction
or even use Stinson's algorithm to construct one STS. Clearly, this STS
fixes the connected component of the transition graph. In the experiments of the
current work, the initial STS was constructed with Stinson's algorithm. If there
would be more than one connected component for some parameters, this would be a dilemma
for Cameron's algorithm, as the choice of starting points would play a central role.
Generating starting points with the right distribution between components seems no easier than
generating uniformly distributed random STSs.

Second, one needs to specify which proper STSs in the Markov chain to output.
We shall get back to the question of convergence to the limiting distribution
in Section~\ref{sec:results}.

\subsection{Modifications of Cameron's Algorithm}

For Cameron's algorithm, it is natural to address the possibility
of considering only proper STSs and perturbations of those. Indeed, it
is shown in~\cite{JM} that the corresponding algorithm for Latin squares
can be modified so that there are no improper intermediate structures.
But a similar result for Steiner triple systems is missing: the
problem of turning Cameron's unmodified algorithm into a Markov
chain whose states are precisely the proper STSs of a fixed order is open.

For (proper) Steiner triple systems one may develop a different algorithm, where in
each step a cycle switch is carried out. Connectivity of the transition
graph is also now the main issue, and here we do know that the graph is
not connected as perfect Steiner triple systems necessarily stay perfect.
But we do arrive at a uniform stationary distribution within a connected
component by defining the switches in the following way.

Choose, uniformly at random, a 2-subset of points $\{a,b\}$ and
a point $x$, $x \not\in \{a,b\}$. If there is a block $\{a,b,x\}$,
then we do nothing; this loop in the transition graph ensures
that the Markov chain is aperiodic. Otherwise we carry out the
cycle switch with $a$ and $b$ as in~\eqref{eq:switch} and
$x \in \{x_1,x_2,\ldots ,x_n\}$. Some switches lead to the same
transition, which can be handled in the framework of multigraphs (or
graphs with weighted edges; this is relevant for the modeling of
probabilities of transitions). It is an interesting question whether
components of the transition graph have odd cycles, which is about
aperiodicity after removing the loops.

\section{Experimental Results}\label{sec:results}

\subsection{Constructing Small Steiner Triple Systems}

As there is a unique \STS{7}, a unique \STS{9}, and two isomorphism classes
of \STSS{13}, the smallest order
that can be used for evaluation of algorithms is 13.
(In~\cite{HDM}, the experimental work was done for orders 15 and 19.)
The two \STSS{13}---here called $S_1$ and $S_2$---have automorphism groups of orders 6 and 39,
that is, there are $13!/6 = 1\,037\,836\,800$ and $13!/39 = 159\,667\,200$
labeled such systems, so $13/15 = 0.8666\ldots$ of the labeled
\STSS{13} belong to the former set and $2/15 = 0.1333\ldots$ to the latter.
The numbers of Pasch configurations in $S_1$ and $S_2$ are 8 and 13, respectively.

To simplify the study of this case, we will here group the labeled STSs into
(the two) isomorphism classes and focus on those.

\subsubsection{Stinson's Algorithm}\label{sec:stinson}

There are two main open questions regarding Stinson's algorithm.
Why does Stinson's algorithm seemingly not generate Steiner triple systems
uniformly at random for small parameters? And how could Stinson's algorithm
be modified to get better distributions for small parameters? A partial answer to
the second question is provided in~\cite{HDM}.

Already Stinson in his seminal paper~\cite{S0} noticed an underrepresentation
of Steiner triple systems with large isomorphism groups in experiments.
Later studies such as~\cite{HDM} mention an underrepresentation of Steiner triple systems
with many Pasch configurations. Notice, however, that
by~\cite[Table~1.29]{MR} there is a correlation between the group order
and the number of Pasch configurations for order 15, a commonly considered
case in experimental studies.

Let us now consider results for variants of Stinson's algorithm.
We study \STSS{13} and denote the number of STSs in an isomorphism class $S$ found in
a set of experiments by $n(S)$. Hence, an algorithm that generates
uniformly distributed \STSS{13} should have
$n(S_1) / (n(S_1) + n(S_2)) \approx 13/15$.
When tabulating such results, we present the percent error
\begin{align*}
100 \cdot \left|\frac{\frac{n(S_1)}{n(S_1)+n(S_2)}-13/15}{13/15}\right|
.
\end{align*}

We modify Stinson's algorithm with five parameters and each modified algorithm is used to generate $10^8$ \STSS{13} in this comparison.
Three parameters are for weighted Stinson's algorithms.
For $i \in \{x,y,z\}$, $W_i = 0$, 1, and 2 mean that the probability
distribution $f_i(j)$ is given by~\eqref{eq:choice} and
$f_i(j) = \sign j$, $f_i(j) = j$, and $f_i(j) = \binom{j}{2}$, respectively.
Stinson's original algorithm corresponds to the case $W_x=W_y=W_z=0$.

Experimental results for weighted Stinson's algorithms are shown in
Table~\ref{tab:weighted}, omitting some cases due to symmetry.

\begin{table}[htbp]
\caption{Weighted Stinson's algorithms}\label{tab:weighted}
\centering
\begin{tabular}{r|rrr|rrr|rrr}
\toprule
$W_x$ & \multicolumn{3}{r|}{0}  & \multicolumn{3}{r|}{1}  & \multicolumn{3}{r}{2} \\
\midrule
$W_y$ & 0 & 1 & 2 & 0 & 1 & 2 & 0 & 1 & 2 \\
\midrule
0 &        3.66 & 3.68 & 3.71 & 3.44 & 3.42 & 3.47 & 2.98 & 3.00 & 3.05 \\
$W_z$ 1 &       & 3.72 & 3.74 &      & 3.48 & 3.47 &      & 3.04 & 3.06 \\
2 &             &      & 3.76 &      &      & 3.46 &      &      & 3.10 \\
\bottomrule
\end{tabular}
\end{table}

The best result is obtained
with the variant $W_x=2$, $W_y=W_z=0$, which cuts roughly
20\% of the deviation of the original algorithm.

The remaining two parameters give different extended Stinson's
algorithms. Different ways of additional switching was considered.
The best result was obtained by carrying out an additional switch
after each execution of the {\bf while} loop of Algorithm~\ref{algo:stinson},
so only the results from that approach will be tabulated here.
As discussed in Section~\ref{sect:switch}, a switch is carried out with respect
to the two points $a$ and $b$. Further consider a point $d$ so that
$\{a,b,d\}$ is \emph{not} a block; the point $d$ gives the cycle
or path in which switching is carried out.

The two parameters for extended Stinson's algorithms are as follows. The values
of $a$, $b$, and $d$ are chosen uniformly with these restrictions.

\begin{itemize}
\item[$O$] $|\{x,y,z\} \cap \{a,b,d\}| \ge O$
\item[$I$] No impact for $I=0$, $y=a$ for $I=1$, and $y=d$ for $I=2$
\end{itemize}

Note that if $I \ne 0$, then $|\{x,y,z\} \cap \{a,b,d\}| \ge 1$ and the cases
$O=0$ and $O=1$ coincide; we therefore omit the latter case.

Table~\ref{tab:weightedandextended} shows computational results for the ten best performing parameter settings.
The overall best variant of Table~\ref{tab:weightedandextended} cuts as much as 99.95\% of the deviation of the original algorithm.
Note that although the results for weighted Stinson's algorithms in Table~\ref{tab:weighted} suggest that it is beneficial to set $W_x=2$, the best variant with this property is just fourth in the list.
All other tested variants were inferior to those presented here.

\begin{table}[htbp]
\caption{Weighted and extended Stinson's algorithms}\label{tab:weightedandextended}
\setlength{\tabcolsep}{5pt}
\centering
\begin{tabular}{rrrrr|l}
\toprule
$W_x$ & $W_y$ & $W_z$ & $O$ & $I$ & percent error\\
\midrule
0 & 0 & 2 & 2 & 2 & 0.002 \\
1 & 0 & 2 & 2 & 2 & 0.005 \\
0 & 1 & 1 & 0 & 2 & 0.011 \\
2 & 1 & 2 & 0 & 2 & 0.013 \\
2 & 1 & 0 & 1 & 0 & 0.014 \\
0 & 2 & 0 & 0 & 2 & 0.016 \\
2 & 0 & 2 & 2 & 2 & 0.019 \\
2 & 1 & 0 & 0 & 2 & 0.020 \\
0 & 2 & 1 & 0 & 2 & 0.031 \\
1 & 1 & 0 & 2 & 2 & 0.037 \\
\bottomrule
\end{tabular}
\end{table}

A major drawback with extended Stinson's algorithms is that the empirical run time
increases by a factor of $v$ to $\Theta(v^3 \log v)$. One remedy for this increase
could be to carry out cycle switching only when the STS is almost complete.

Returning to the question about why Stinson's algorithm behaves as it does,
we present data on partial \STSS{13} in Tables~\ref{tab:partial13a}
and~\ref{tab:partial13b}, the latter also containing some experimental results.
In Table~\ref{tab:partial13a}, we give the number of blocks $k$ and the number
of labeled partial \STSS{13} that can be completed only to $S_1$, to
both $S_1$ and $S_2$, and only to $S_2$ in columns
$N_1$, $N_{12}$, and $N_2$, respectively. Finally, we give
the total number of isomorphism classes in the aforementioned three groups $P'$
as well as the total number of isomorphism classes of partial \STSS{13} $P$.
The algorithms used to get these numbers are standard~\cite{KO2}: removing blocks
from \STSS{13} and building up partial \STSS{13} block by block, with isomorph
rejection.

\begin{table}[htbp]
\caption{Partial STS(13)s}\label{tab:partial13a}
\centering
{\footnotesize
\begin{tabular}{rrrrrrrrr}\toprule
$k$ & $N_1$ & $N_{12}$ & $N_2$ & $P'$ & $P$\\\midrule
 1& 0               & 286             & 0                & 1      & 1         \\
 2& 0               & 36465           & 0                & 2      & 2         \\
 3& 0               & 2748460         & 0                & 5      & 5         \\
 4& 0               & 136963255       & 0                & 16     & 16        \\
 5& 2522520         & 4777742970      & 0                & 53     & 54        \\
 6& 539458920       & 119432267955    & 0                & 232    & 250       \\
 7& 50727877200     & 2112851941200   & 0                & 1259   & 1419      \\
 8& 2867932267200   & 24450299874000  & 0                & 7843   & 9768      \\
 9& 75719643199200  & 152456669164800 & 272691619200     & 46338  & 71311     \\
10& 657490838724720 & 483068044499040 & 10805956761600   & 201856 & 482568    \\
11& 2411136555347520& 893688887931840 & 86263438060800   & 562120 & 2729981   \\
12& 5003137094155200& 1109870976614400& 295345520870400  & 1041051& 12142983  \\
13& 7144211666592000& 1012594014432000& 589917494073600  & 1411388& 41023224  \\
14& 7818601132742400& 711933731308800 & 810269287027200  & 1504182& 103043009 \\
15& 6910864959398400& 394063519449600 & 834134344243200  & 1310150& 189057254 \\
16& 5054265734918400& 172910875737600 & 674036601638400  & 950220 & 248583304 \\
17& 3088955181312000& 59952718425600  & 438749658547200  & 578055 & 228896680 \\
18& 1580512841107200& 16199594611200  & 233224761369600  & 295188 & 143386618 \\
19& 674387390476800 & 3313812902400   & 101714233420800  & 126036 & 58893945  \\
20& 237726897408000 & 485707622400    & 36274471833600   & 44573  & 15142370  \\
21& 68154742656000  & 45664819200     & 10457243596800   & 12875  & 2306220   \\
22& 15510470976000  & 2075673600      & 2384948966400    & 2994   & 197746    \\
23& 2698375680000   & 0               & 415134720000     & 548    & 9348      \\
24& 337296960000    & 0               & 51891840000      & 76     & 267       \\
25& 26983756800     & 0               & 4151347200       & 10     & 10        \\
26& 1037836800      & 0               & 159667200        & 2      & 2         \\\bottomrule
\end{tabular}
}
\end{table}

Table~\ref{tab:partial13b} contains some of the data from
Table~\ref{tab:partial13a} in processed form as well as some computational
results. Specifically, $p_X := N_X/(N_1+N_{12}+N_2)$.
In the columns $q_X$, we show the distribution of partial
STSs obtained experimentally with Algorithm~\ref{algo:stinson}
in $10^6$ runs. Exact
values of $q_X$ could also be obtained analytically in the context of
absorbing Markov chains, but such an approach seems impracticable here
due to the very high number of states (partial STSs).

\begin{table}[htbp]
\caption{Data for partial STS(13)s}\label{tab:partial13b}
\begin{center}
\setlength{\tabcolsep}{5pt}
\begin{tabular}{r@{\hskip 0.9cm}ccc@{\hskip 0.9cm}ccc}\toprule
$k$ & $p_1$ & $p_{12}$ & $p_2$ & $q_1$ & $q_{12}$ & $q_2$\\\midrule
 1 & 0.0000 &1.0000 &0.0000 & 0.0000 &1.0000 &0.0000\\
 2 & 0.0000 &1.0000 &0.0000 & 0.0000 &1.0000 &0.0000\\
 3 & 0.0000 &1.0000 &0.0000 & 0.0000 &1.0000 &0.0000\\
 4 & 0.0000 &1.0000 &0.0000 & 0.0000 &1.0000 &0.0000\\
 5 & 0.0005 &0.9995 &0.0000 & 0.0005 &0.9995 &0.0000\\
 6 & 0.0045 &0.9955 &0.0000 & 0.0043 &0.9957 &0.0000\\
 7 & 0.0234 &0.9766 &0.0000 & 0.0221 &0.9779 &0.0000\\
 8 & 0.1050 &0.8950 &0.0000 & 0.0913 &0.9087 &0.0000\\
 9 & 0.3315 &0.6674 &0.0012 & 0.2775 &0.7219 &0.0006\\
10 & 0.5711 &0.4196 &0.0094 & 0.4898 &0.5054 &0.0048\\
11 & 0.7110 &0.2635 &0.0254 & 0.6267 &0.3599 &0.0134\\
12 & 0.7807 &0.1732 &0.0461 & 0.7088 &0.2666 &0.0246\\
13 & 0.8168 &0.1158 &0.0674 & 0.7627 &0.1989 &0.0384\\
14 & 0.8370 &0.0762 &0.0867 & 0.7977 &0.1488 &0.0535\\
15 & 0.8491 &0.0484 &0.1025 & 0.8194 &0.1107 &0.0699\\
16 & 0.8565 &0.0293 &0.1142 & 0.8365 &0.0786 &0.0849\\
17 & 0.8610 &0.0167 &0.1223 & 0.8437 &0.0572 &0.0991\\
18 & 0.8637 &0.0089 &0.1274 & 0.8540 &0.0361 &0.1099\\
19 & 0.8652 &0.0043 &0.1305 & 0.8587 &0.0223 &0.1190\\
20 & 0.8661 &0.0018 &0.1322 & 0.8644 &0.0119 &0.1237\\
21 & 0.8665 &0.0006 &0.1329 & 0.8656 &0.0051 &0.1293\\
22 & 0.8666 &0.0001 &0.1333 & 0.8717 &0.0020 &0.1263\\
23 & 0.8667 &0.0000 &0.1333 & 0.8762 &0.0000 &0.1238\\
24 & 0.8667 &0.0000 &0.1333 & 0.8828 &0.0000 &0.1172\\
25 & 0.8667 &0.0000 &0.1333 & 0.8985 &0.0000 &0.1015\\
26 & 0.8667 &0.0000 &0.1333 & 0.8985 &0.0000 &0.1015\\\bottomrule
\end{tabular}
\end{center}
\end{table}

Notably, the skew distribution emerges in late stages. For
example, the experimentally obtained distribution for 21-block partial
STSs is close to the theoretically obtained one.

\subsubsection{Cameron's Algorithm}

Considering the case of \STSS{13},
the subsequence of isomorphism classes of proper STSs encountered by
Cameron's algorithm can be modeled as a Markov chain with two states, which
we also call $S_1$ and $S_2$.
We denote the transition probabilities
$S_1 \rightarrow S_2$ and $S_2 \rightarrow S_1$ by $p$ and $q$, respectively,
and the limiting, stationary distribution by $(\pi(1),\pi(2))$.
Determining $p$ and $q$ analytically seems challenging, but we can get
estimates with computational experiments. In a Markov chain with $10^9$
proper STSs, we get estimates $p \approx 0.108954$ and $q \approx 0.708279$.
A two-state Markov chain with such transition probabilities
converges to a probability
distribution $(q/(p+q) \approx 0.86668$, $p/(p+q)\approx0.13332)$, which is
reasonably close to $(\pi(1)=13/15,\pi(2)=2/15)$.

The values of $p$ and $q$ can also be used to determine the convergence
rate. Let $\mu_t(i)$ denote the probability that we are in state
$i$ at time $t$. Regardless of the initial state,
\[
|\mu_t(i)-\pi(i)|<|1-p-q|^t<0.2^t\mbox{\ for\ }i \in \{1,2\};
\]
see, for example,~\cite{R}.
That is, we have exponential convergence.
In fact, all Markov chains that are indecomposable and aperiodic
converge exponentially quickly~\cite{R}. Unfortunately, a more formal
treatment gets difficult for our problem with STSs of order greater than 13.

\subsection{Constructing Large Steiner Triple Systems}

The fact that Stinson's algorithm apparently does not
perform optimally for \STSS{13} and other small parameters does not
necessarily mean that the same holds for larger parameters.
Although we do conjecture that Stinson's algorithm performs nonoptimally
with respect to certain individual STSs for any order,
we do not exclude the possibility that the obtained distribution
converges in some manner to the desired distribution. (We refrain
from formulating this statement explicitly, which would also require
picking an appropriate measure for comparing distributions.)
Nevertheless we do have some experimental results showing that the situation
may be different for small and large orders.

We implemented Stinson's algorithm and Cameron's algorithm and carried out
a set of experiments for a wide range of orders. In Cameron's algorithm,
every $v$th proper \STS{v} was output. For each order up to 200, $10^6$
random STSs were generated and the numbers of the configurations listed
in Figure~\ref{fig:all_configurations} were counted using algorithms from~\cite{HO2}.
The results after dividing by the right-hand side of~\eqref{eq:asymptotic} are
shown in Figures~\ref{fig:subP} to~\ref{fig:sub8}.
The continuous lines in Figures~\ref{fig:subP} to~\ref{fig:subG} show the
minimum, mean, and maximum values of the underlying distribution.
The boxplots show the first quartile, median, and third quartile, and the whiskers show
the furthest datapoints of the distribution which fall into $1.5$ times
the interquartile range. As most STSs contain neither Fano nor M\"{o}bius--Kantor
configurations, only the mean value divided by the right-hand side
of~\eqref{eq:asymptotic} is presented in Figures~\ref{fig:sub7}
and~\ref{fig:sub8}.

It can be seen that the degree of the monomials in Table~\ref{tab:configuration_details}
has an impact on convergence. The cases where it is not 2 are
Fano--line (degree 1; Figure~\ref{fig:subF}) and prism and grid
(degree 3; Figures~\ref{fig:subR} and~\ref{fig:subG}).
For the Pasch, mitre, crown, and hexagon, the curves seem to converge
very nicely to 1. These are precisely the configurations with a
quadratic polynomial in Table~\ref{tab:configuration_details}.
Also in the cases of Fano--line, prism, grid, Fano,
and M\"obius--Kantor, 1 is approached, but a consideration of larger parameters
would be needed to study the behavior more closely.

The only noticeable difference between Stinson's algorithm and Cameron's algorithm is
for Fano planes in STSs of small order (Figure~\ref{fig:sub7}). The curves cannot be
used to draw conclusions, but this hints that STSs with Fano planes might be
underrepresented for small parameters when Stinson's algorithm is used.

\section{Conclusions}\label{sec:conclusions}

The current study does not show clear evidence in favor of any of the two studied algorithms for
large orders, and both have their advantages and disadvantages. One disadvantage
of Cameron's algorithm is that one might end up having to implement both algorithms.
The authors hope that the current results inspire work on alternative approaches
for evaluating algorithms for constructing random Steiner triple systems of large orders.
Moreover, the experimental results linking the hypergraph models and the behavior of the
algorithms motivate theoretical work in this area.

\begin{figure}
\centering
\ImgPasch
\caption{Pasch}\label{fig:subP}
\end{figure}

\begin{figure}
\centering
\ImgMitre
\caption{Mitre}\label{fig:subM}
\end{figure}

\begin{figure}
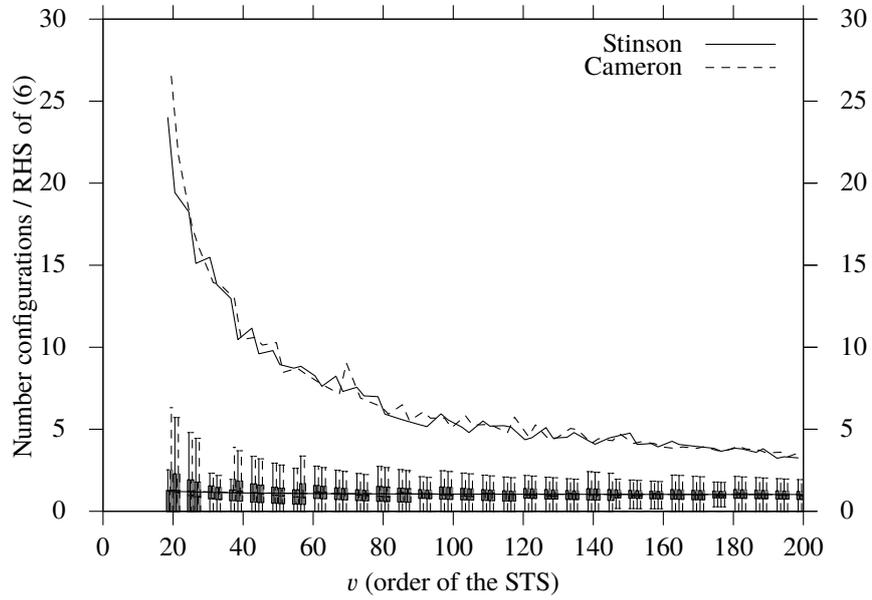

\centering
\ImgFanoLine
\caption{Fano--line}\label{fig:subF}
\end{figure}

\begin{figure}
\centering
\ImgCrown
\caption{Crown}\label{fig:subC}
\end{figure}

\begin{figure}
\centering
\ImgHexagon
\caption{Hexagon}\label{fig:subH}
\end{figure}

\begin{figure}
\centering
\ImgPrism
\caption{Prism}\label{fig:subR}
\end{figure}

\begin{figure}
\centering
\ImgGrid
\caption{Grid}\label{fig:subG}
\end{figure}

\begin{figure}
\centering
\ImgFano
\caption{Fano}\label{fig:sub7}
\end{figure}

\begin{figure}
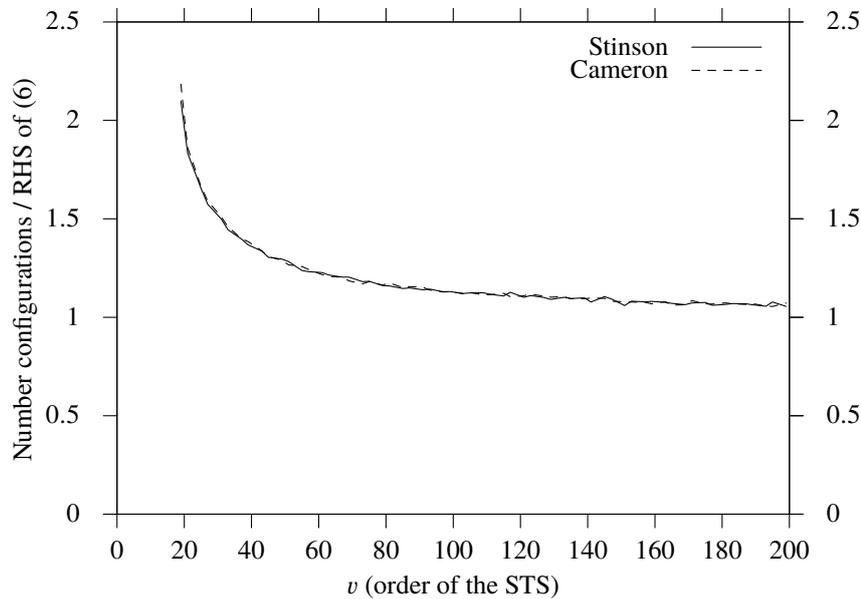

\centering
\ImgMobiusKantor
\caption{M\"{o}bius--Kantor}\label{fig:sub8}
\end{figure}

\section*{Acknowledgements}
The authors are grateful to the anonymous referees for valuable comments.

\end{document}